\documentclass[number]{elsarticle}
\usepackage[hidelinks, colorlinks=true]{hyperref}
\usepackage{amscd,amssymb,amsmath,amsthm}
\usepackage[all]{xy}
\usepackage{array}
\usepackage{etoolbox}
\usepackage{mathtools}
\DeclarePairedDelimiter{\ceil}{\lceil}{\rceil}
\DeclarePairedDelimiter{\floor}{\lfloor}{\rfloor}
\makeatletter
\def\ps@pprintTitle{%
 \let\@oddhead\@empty
\let\@evenhead\@empty
\def\@oddfoot{\reset@font\hfil\thepage\hfil}
\let\@evenfoot\@oddfoot

}

\makeatother

\newdir{ >}{!/8pt/\dir{}*\dir{>}}

\newtheorem*{th26}{Theorem 2.6}
\newtheorem*{th31}{Theorem 3.1}
\newtheorem*{th35}{Theorem 3.5}
\newtheorem*{cr32}{Corollary 3.2}
\newtheorem*{cr33}{Corollary 3.3}
\newtheorem*{th25}{Theorem 2.5}
\newtheorem*{p1}{Problem 1}
\newtheorem*{p2}{Problem 2}
\newtheorem{theorem}{Theorem}[section]

\newtheorem{lemma}[theorem]{Lemma}
\newtheorem{corollary}[theorem]{Corollary}
\newtheorem{prop}[theorem]{Proposition}

\newtheorem{con}{Conjecture }

\numberwithin{equation}{theorem}

\theoremstyle{definition}

\DeclareMathOperator{\e}{\mathit{exp}}

\journal{}

\begin{document}

\begin{frontmatter}

\title{Bounding the exponent of a finite group by the exponent of the automorphism group and a theorem of Schur}

\author[IISER TVM]{P. Komma}
\ead{patalik16@iisertvm.ac.in}
\author[IISER TVM]{V.Z. Thomas\corref{cor1}}
\address[IISER TVM]{School of Mathematics,  Indian Institute of Science Education and Research Thiruvananthapuram,\\695551
Kerala, India.}
\ead{vthomas@iisertvm.ac.in}
\cortext[cor1]{Corresponding author. \emph{Phone number}: +91 8921458330}

\begin{abstract}
Assume $G$ is a finite $p$-group, and let $S$ be a Sylow $p$-subgroup of $\operatorname{Aut}(G)$ with $\exp(S)=q$. We prove that if $G$ is of class $c$, then $\exp(G)|p^{\ceil{\log_pc}}q^3$, and if $G$ is a metabelian $p$-group of class at most $2p-1$, then $\exp(G)|pq^3$. 
\end{abstract}

\begin{keyword}
 Schur Multiplier  \sep automorphisms of groups
\MSC[2010]   20B05  \sep 20D15  \sep 20F14 \sep 20F18 \sep 20G10 \sep 20J05 \sep 20J06 
\end{keyword}

\end{frontmatter}

 \section{Introduction} 

 The study of the relationship between $|G|$ and $|\operatorname{Aut}(G)|$ has attracted many researchers (see \cite{BH}, \cite{JEA} and \cite{PSY}). In \cite[Theorem 1]{EKS}, the authors bound the exponent of a finite group $G$ with automorphisms. Motivated by this, we propose the following problem:

 \begin{p1}
Let $G$ be a finite group. Can we find a function $f_1:\mathbb{N}\rightarrow \mathbb{N}$ such that $\exp(G)\mid f_1(\exp(\operatorname{Aut}(G)))$, and can we describe $f_1$ explicitly for certain classes of groups?
\end{p1}

 A classical theorem of Schur \cite{IS} states that if the central quotient $G/Z(G)$ is finite, then $\gamma_2(G)$ is finite. In \cite[Theorem 1]{AM1}, the author proves that if $G$ is a group in which $G/Z(G)$ is locally finite of exponent $n$, then $\gamma_2(G)$ is locally finite, and has finite exponent, bounded by a function $f(n)$ depending only on $n$. With this we state the next problem.

\begin{p2}
Let $G$ be a finite group. Can we find a function $f_2:\mathbb{N}\rightarrow \mathbb{N}$ such that $\exp(\gamma_2(G))\mid f_2(\exp(G/Z(G)))$, and can we describe $f_2$ explicitly for certain classes of groups?
\end{p2}

In \cite{APTIsrael}, the authors prove that for finite $p$-groups of class at most $p+1$, $f_2$ can be taken to be the identity. One of the main objectives of the present paper is to answer Problem $2$ and show that for finite metabelian $p$-groups of class at most $2p-1$, $f_2$ can be taken to be the identity, and for finite $p$-groups of class $c$, $f_2$ can be taken as $f_2(n)=p^{n+\ceil{\log_pc}-1}$. In particular, we prove:

\begin{th26}
Let $p$ be an odd prime and $G$ be a finite metabelian $p$-group. If the nilpotency class of $G$ is less than or equal to $2p-1$, then $\exp(\gamma_2(G))\mid \exp(G/Z(G))$.
\end{th26}

\begin{th31}
Let $p$ be a prime and $G$ be a $p$-group. If the nilpotency class of $G$ is $c$, then $\exp(\gamma_{2}(G))\mid p^{\ceil{\log_pc}-1}\exp(G/Z(G))$.
\end{th31}

Using these bounds, we answer Problem $1$ for metabelian $p$-groups of class at most $2p-1$, and finite $p$-groups of a given nilpotency class. In particular, we bound $\exp(G)$ by a function of the exponent of a Sylow $p$-subgroup of $\operatorname{Aut}(G)$, and prove

\begin{th35}
Let $G$ be a finite $p$-group and $S$ be a Sylow $p$-subgroup of $\operatorname{Aut}(G)$ with $\exp(S)=q$.
\begin{itemize} 
\item [$(i)$] If $G$ is metabelian $p$-group of class at most $2p-1$, then $\exp(G)\mid pq^3$.
\item [$(ii)$] If the nilpotency class of $G$ is $c$, then $\exp(G)\mid p^{\ceil{\log_pc}}q^3$.
\end{itemize}
\end{th35}

Let $G$ be a finite group, then Schur's exponent problem states that $\exp(H_2(G,\mathbb{Z})) \mid \exp(G)$.  In \cite{BKW}, the authors found a counterexample to this problem. Their counterexample involved a $2$-group of order $2^{68}$ with $\exp(G)=4$ and $\exp(H_2(G,\mathbb{Z}))=8$. A detailed account of this problem can be found in \cite{VT1}. More recently, the author of \cite{MVL} found counterexamples to Schur's exponent problem for odd order groups. He found a $5$-group $G$ of class $9$ such that $\exp(G)=5$ and $\exp(H_2(G, \mathbb{Z}))=25$, as well as a $3$-group $G$ of class $9$ such that $\exp (G)=9$ and $\exp(H_2(G, \mathbb{Z}))=27$. All of these counterexamples satisfy $\exp(H_2(G,\mathbb{Z})) \mid p \exp(G)$. In \cite{CM}, Miller proved that $H_2(G, \mathbb{Z})$ is a normal subgroup of $G\wedge G$. The group $G\wedge G$ is isomorphic with the commutator subgroup $\gamma_2(H)$ of any covering group $H$ of $G$. Noting that $\exp(\gamma_2(H))\mid f(\exp(H/Z(H)))$ for every group $H$ of class $c$ is equivalent to $\exp(G\wedge G)\mid f(\exp(G))$ for every group $G$ of class $c-1$, Theorem \ref{th:3.1} immediately yields:

\begin{cr32}
Let $p$ be an odd prime and $G$ be a finite $p$-group. If the nilpotency class of $G$ is $c$, then $\exp(G\wedge G)\mid p^{n-1}\exp(G)$, for $n=\ceil{\log_p(c+1)}$. In particular, $\exp(H_2(G, \mathbb{Z}))\mid p^{n-1}\exp(G)$.
\end{cr32}

\begin{cr33}
Let $p$ be an odd prime and $G$ be a finite $p$-group. If the nilpotency class of $G$ is less than or equal to $p^2-1$, then $\exp(G\wedge G)\mid p\exp(G)$. In particular, $\exp(H_2(G, \mathbb{Z}))\mid p\exp(G)$.
\end{cr33}

With this evidence in hand, we conjecture:

\begin{con}
If $G$ is a finite $p$ group, then $\exp(H_2(G, \mathbb{Z}))\mid p\exp(G)$.
\end{con}

It is not known if $\exp(H_2(G, \mathbb{Z}))\mid \exp(G)$ for metabelian groups. Our contribution towards this question is the following theorem:

\begin{th25}\label{th:2.5}
Let $p$ be an odd prime and $G$ be a finite metabelian $p$-group. If the nilpotency class of $G$ is less than or equal to $2p-1$, then $\exp(G\wedge G)\mid \exp(G)$. In particular, $\exp(H_2(G, \mathbb{Z}))\mid \exp(G)$.
\end{th25}

\section{Schur's exponent conjecture for metabelian $p$-groups of class $2p-1$}

Let $G$ be a finite $p$-group with $\exp(G)=p^n$. Note that if $G$ belongs to the class of regular $p$-groups, powerful $p$-groups or potent $p$-groups, then $G^p$ is powerful and $\exp(G^p)=p^{n-1}$. In this section, we prove the same for metabelian $p$-groups of class at most $2p-1$. As a consequence, we get the Schur's exponent conjecture for metabelian $p$-groups of class $2p-1$. We begin with the following theorem:

\begin{theorem}\label{th:2.1}(See \cite[Theorem 2.4 and Theorem 2.5]{ASZ})
Let $G$ be a finite $p$-group. For $N, M$, normal subgroups of $G$ we have
\begin{itemize}
\item [$(i)$] $[N^{p^n}, M] \equiv [N, M]^{p^n} \mod \prod_{r=1}^n [M,\ _{p^r}\ N]^{p^{n-r}}$. 
\item [$(ii)$] $[N^{p^n},\ _l\ G] \equiv [N,\ _l\ G]^{p^n} \mod \prod_{r=1}^{n} [N,\ _{p^r+l-1}\ G]^{p^{n-r}}$.
\end{itemize}
\end{theorem}

In the next proposition, we obtain upper bounds for the exponent of certain subgroups of a finite metabelian $p$-group of class at most $2p-1$.

\begin{prop}\label{P:2.2}
Let $p$ be an odd prime and $G$ be a finite metabelian $p$-group of class less than or equal to $2p-1$. If $\exp(G)=p^{n}\ge p^2$, then 
\begin{itemize}
\item [$(i)$] $\exp(\gamma_2(G^p))\mid p^{n-1}$, 
\item [$(ii)$] $\exp(\gamma_3(G^p))\mid p^{n-2}$.
\end{itemize}
\end{prop}

\begin{proof}
By Theorem \ref{th:2.1} $(i)$, we obtain
\begin{equation*}
[G^p, G^p]\le [G, G^p]^p [G^p,\ _p\ G]
\end{equation*}
Since $\gamma_{2p}(G)=1$, using Theorem \ref{th:2.1} $(ii)$ yields $[G^p,\ _p\ G]=(\gamma_{p+1}(G))^p$. Moreover, $\gamma_{p+1}(G)\le [G^p, G]$ by \cite[Theorem 13]{LCKM}. Thus, we get 
\begin{equation}\label{eq:2.2.1}
[G^p, G^p]\le [G^p, G]^p.
\end{equation}
Hence $(\gamma_2(G^p))^{p^{n-1}}\le [G^p, G]^{p^n}=1$. Note that, as $G$ is metabelian, $[N^p, M]=[N, M]^p$ for normal subgroups $N, M$ of $G$, where $N\le \gamma_2(G)$. Thus, using \eqref{eq:2.2.1} for $[G^p, G^p]$ in $[G^p, G^p, G^p]$, we obtain 
\begin{equation*}
[G^p, G^p, G^p]\le [[G^p, G]^p, G^p]^p=[G^p, G, G^p]^p.
\end{equation*}
Hence $(\gamma_3(G^p))^{p^{n-2}}\le (\gamma_2(G^p))^{p^{n-1}}=1$.
\end{proof}

\begin{corollary}\label{cr:2.3}
Let $p\ge 3$ and $G$ be a metabelian $p$-group of class at most $2p-1$. If $\exp(G)=p^n$, then $\exp(G^p)=p^{n-1}$.
\end{corollary}

\begin{proof}
We may assume $n\ge 2$. To prove $\exp(G^p)=p^{n-1}$, we prove that $(x_1^p\cdots x_k^p)^{p^{n-1}}=1$ for all $x_1, \ldots x_k\in G$, $k\ge 2$. Using Hall's commutator collection formula, we have
\begin{equation*}
(x_1^p\cdots x_k^p)^{p^{n-1}}\equiv x_1^{p^n}\cdots x_k^{p^n} \mod (\gamma_2(G^p))^{p^{n-1}} (\gamma_p(G^p))^{p^{n-2}}.
\end{equation*}
Now by Proposition \ref{P:2.2}, we obtain $(x_1^p\cdots x_k^p)^{p^{n-1}}=1$.
\end{proof}

Now we recall \cite[Lemma 3.9]{APTIsrael}.
\begin{lemma}\label{L:2.4}
Let $p$ be an odd prime and $G$ be a finite $p$-group with $\exp(G)=p^n$. Suppose $G$ satisfies the following conditions:
\begin{itemize}
\item [$(i)$] $G^{p}$ is powerful.
\item[$(ii)$] $\exp(G^{p})=p^{n-1}$.
\item [$(iii)$] $\gamma_{p+1}(G)\le G^{p}$.
\end{itemize}
Then $\exp(G\wedge G)\mid \e(G)$. In particular, $\exp(H_2(G, \mathbb{Z}))|\exp(G)$.
\end{lemma}

\begin{theorem}\label{th:2.5}
Let $p$ be an odd prime and $G$ be a finite metabelian $p$-group. If the nilpotency class of $G$ is less than or equal to $2p-1$, then $\exp(G\wedge G)\mid \exp(G)$. In particular, $\exp(H_2(G, \mathbb{Z}))\mid \exp(G)$.
\end{theorem}

\begin{proof}
Let $\exp(G)=p^n$, we will show that $G$ satisfies the hypothesis of Lemma \ref{L:2.4}. If $\exp(G)=p$, then $G$ has class at most $p$ by  \cite[Satz.\ 3]{MW}. Hence the theorem holds by \cite[Theorem 2.4]{APTIsrael}. Now we may assume $n\ge 2$. It follows that $G^p$ is powerful by \eqref{eq:2.2.1}, and $\exp(G^p)=p^{n-1}$ by Corollary \ref{cr:2.3}. Moreover, \cite[Satz.\ 3]{MW} yields $\gamma_{p+1}(G)\le G^p$ as required. 
\end{proof}

\begin{theorem}\label{th:2.6}
Let $p$ be an odd prime and $G$ be a finite metabelian $p$-group. If the nilpotency class of $G$ is less than or equal to $2p-1$, then $\exp(\gamma_2(G))\mid \exp(G/Z(G))$.
\end{theorem}

\begin{proof}
Let $\exp(G/Z(G))=p^n$. If $n=1$, then we have $(\gamma_2(G))^p=1$ by \cite[Theorem 13]{LCKM}. Now we assume $n\ge 2$. Using Theorem \ref{th:2.1} $(i)$, we have $(\gamma_2(G))^{p^n}\le [G^{p^n}, G] (\gamma_{p+1}(G))^{p^{n-1}}$. Applying \cite[Theorem 13]{LCKM} to $G/[G^p, G]$, we obtain $\gamma_{p+1}(G)\le [G^p, G]$. Thus we get $(\gamma_2(G))^{p^n}\le [G^p, G]^{p^{n-1}}$. Moreover, 
\begin{equation}\label{eq:2.6.1}
[G^p, G]^{p^{n-1}}\le [(G^p)^{p^{n-1}}, G] [G,\ _{p}\ G^p]^{p^{n-2}}
\end{equation}
by Theorem \ref{th:2.1} $(i)$. Applying Corollary \ref{cr:2.3} to $\frac{G}{G^{p^n}}$, we get $(G^p)^{p^{n-1}}\le G^{p^n}$. Hence $[(G^p)^{p^{n-1}}, G]=1$. Set $L=[G,\ _{p-1}\ G^p]$. By \eqref{eq:2.6.1}, $[G^p, G]^{p^{n-1}}\le [L, G^p]^{p^{n-2}}$. Noting that $[L,\ _p\ G]\le \gamma_{2p}(G)$, Theorem \ref{th:2.1} $(i)$ yields $[L, G^p]=[L, G]^p$. As $G$ is metabelian, we have $([L, G]^p)^{p^{n-2}}=[L, G]^{p^{n-1}}= [(L)^{p^{n-1}}, G]$. Thus we get $[G^p, G]^{p^{n-1}}\le [L^{p^{n-1}}, G]$. Applying Proposition \ref{P:2.2} $(i)$ to $\frac{G}{G^{p^n}}$, we obtain $(\gamma_2(G^p))^{p^{n-1}}\le G^{p^n}$. As $L\le \gamma_2(G^p)$, we get $L^{p^{n-1}}\le G^{p^n}$, yielding that $[L^{p^{n-1}}, G]=1$. Hence $[G^p, G]^{p^{n-1}}=1$, as required.
\end{proof}

\section{Bounds depending on the nilpotency class}
A classical theorem of Schur states that if the central quotient $G/Z(G)$ is finite, then $\gamma_2(G)$ is finite. In \cite{APTIsrael}, the authors prove that if $G$ is a finite $p$-group of class at most $p+1$, then $\exp(\gamma_2(G))\mid \exp(G/Z(G))$, which can be regarded as an analogue of Schur's theorem for the exponent of the group instead of the order. In this section, we give a bound for $\exp(\gamma_2(G))$ for $p$-groups of a given class. In particular, we prove the following theorem:

\begin{theorem}\label{th:3.1}
Let $p$ be a prime and $G$ be a $p$-group. If the nilpotency class of $G$ is $c$, then $\exp(\gamma_{2}(G))\mid p^{\ceil{\log_pc}-1}\exp(G/Z(G))$. 
\end{theorem}

\begin{proof}
Let $\exp(G/Z(G))=p^n$. We show that 

\begin{equation}\label{eq:3.1.1}
\exp(\gamma_{i+1}(G))\mid p^{n+\ceil{\log_p(\frac{c}{i})}-1}
\end{equation}

 for all $1\le i\le c-1$. If $c\le p+1$, then \eqref{eq:3.1.1} follows by \cite[Theorem 2.4]{APTIsrael}, so we assume $c\ge p+2$. Let $m=\ceil{\frac{c}{p}}$. Noting that $\gamma_{pm+1}(G)\le \gamma_{c+1}(G)$, using Theorem \ref{th:2.1} $(i)$, we get $(\gamma_{m+1}(G))^{p^n}=[(\gamma_m(G)^{p^n}), G]=1$. Thus \eqref{eq:3.1.1} holds for all $m\le i\le c-1$. Now we proceed to prove \eqref{eq:3.1.1} by reverse induction on $i$. Let $1\le i\le m-1$, noting that  $\gamma_{ip^{\ceil{\log_p(\frac{c}{i})}}+1}(G)\le \gamma_{c+1}(G)$, we have $\prod_{r=\ceil{\log_p(\frac{c}{i})}}^{n+\ceil{\log_p(\frac{c}{i})}-1} (\gamma_{ip^r+1}(G))^{p^{n+\ceil{\log_p(\frac{c}{i})}-1-r}} =1$. Thus using Theorem \ref{th:2.1} $(i)$ yields 

 \begin{equation*}
 (\gamma_{i+1}(G))^{p^{n+\ceil{\log_p(\frac{c}{i})}-1}}\le [(\gamma_i(G))^{p^{n+\ceil{\log_p(\frac{c}{i})}}-1}, G] \prod_{r=1}^{\ceil{\log_p(\frac{c}{i})}-1} (\gamma_{ip^r+1}(G))^{p^{n+\ceil{\log_p(\frac{c}{i})}-1-r}}.
 \end{equation*}

 By induction hypothesis, $\exp(\gamma_{ip^r+1}(G))\mid p^{n+\ceil{\log_p\bigg(\dfrac{c}{ip^r}\bigg)}-1}=p^{n+\ceil{\log_p(\frac{c}{i})}-1-r}$ for all $1\le r\le \ceil{\log_p(\frac{c}{i})}-1$. Now noting that $[(\gamma_i(G))^{p^{n+\ceil{\log_p(\frac{c}{i})}-1}}, G]\le [(\gamma_i(G))^{p^n}, G]=1$, we get \eqref{eq:3.1.1} for $i$. In particular, by taking $i=1$, we have $\exp(\gamma_{2}(G))\mid p^{n+\ceil{\log_pc}-1}$.
\end{proof}

In \cite[Theorem 2.4]{APTIsrael} the authors proved that $\exp(\gamma_2(G))\mid \exp(G/Z(G))$ for finite $p$-group $G$ of class at most $p+1$, where as the above theorem gives that $\exp(\gamma_2(G))\mid \exp(G/Z(G))$ if the class of $G$ is at most $p$, and $\exp(\gamma_2(G))\mid p\exp(G/Z(G))$ if the class of $G$ is $p+1$. Let $p$ be an odd prime and $G$ be a nilpotent group of class $c$. Sambonet \cite[Theorem 1.1]{NS2} proved that $\exp(G\wedge G)\mid (\exp(G))^{\floor{\log_{p-1}c}+1}$. The authors \cite[Theorem 4.2]{APTIsrael} proved that $\exp(G\wedge G)\mid (\exp(G))^{\ceil{\log_{p-1}(\frac{c+1}{p+1})}+1}$, if $c\ge p$. The authors \cite[Theorem 1.4]{BMGM} proved that $\exp(G\wedge G)\mid (\exp(G))^{\ceil{\log_p(c+1)}}$. As a consequence of Theorem \ref{th:3.1}, we obtain the following result, which improves the bounds given by \cite[Theorem 1.4]{BMGM}, and consequently the bounds given by \cite[Theorem 4.2]{APTIsrael}, and \cite[Theorem 1.1]{NS2}.

\begin{corollary}\label{cr:3.2}
Let $p$ be an odd prime and $G$ be a finite $p$-group. If the nilpotency class of $G$ is $c$, then $\exp(G\wedge G)\mid p^{n-1}\exp(G)$, for $n=\ceil{\log_p(c+1)}$. In particular, $\exp(H_2(G, \mathbb{Z}))\mid p^{n-1}\exp(G)$.
\end{corollary}

\begin{corollary}\label{cr:3.3}
Let $p$ be an odd prime and $G$ be a finite $p$-group. If the nilpotency class of $G$ is less than or equal to $p^2-1$, then $\exp(G\wedge G)\mid p\exp(G)$. In particular, $\exp(H_2(G, \mathbb{Z}))\mid p\exp(G)$.
\end{corollary}

For an odd order group of nilpotency class at most $8$, we have the following bound for the exponent of the Schur multiplier.

\begin{corollary}\label{cr:3.4}
Let $p$ be an odd prime and let $G$ be a $p$-group. If the nilpotency class of $G$ is less than or equal to $8$, then $\exp(G\wedge G) \mid\ p\exp(G)$. In particular, $\exp(H_2(G, \mathbb{Z})) \mid\ p\exp(G)$.
\end{corollary}

\begin{theorem}\label{th:3.5}
Let $G$ be a finite $p$-group and $S$ be a Sylow $p$-subgroup of $\operatorname{Aut}(G)$ with $\exp(S)=q$.
\begin{itemize} 
\item [$(i)$] If $G$ is metabelian $p$-group of class at most $2p-1$, then $\exp(G)\mid pq^3$.
\item [$(ii)$] If the nilpotency class of $G$ is $c$, then $\exp(G)\mid p^{\ceil{\log_pc}}q^3$.
\end{itemize}
\end{theorem}

\begin{proof}
\begin{itemize}
\item [$(i)$] Noting that $\exp(G/Z(G))\mid q$, we obtain $\exp(\gamma_2(G))\mid q$ by Theorem \ref{th:2.6}. Following the proof of \cite[Theorem 150A]{{BJ4}} , we get either $\exp(Z(G))\mid q$ or $\exp(G/\gamma_2(G))\mid pq^2$. Thus either $\exp(Z(G))\mid q$ giving $\exp(G)\mid q^2$, or $\exp(G/\gamma_2(G))\mid pq^2$ giving $\exp(G)\mid pq^3$. Hence we get $\exp(G)\mid pq^3$ in either case.

\item [$(ii)$] Using Theorem \ref{th:2.6} in place of Theorem \ref{th:3.1} in the proof of $(i)$, we obtain $(ii)$.
\end{itemize}
\end{proof}

Since $\exp(G)= \prod\limits_{p\mid |G|} \exp(S_p)$ for a finite group $G$, where $S_p$ is a Sylow $p$-subgroup of $G$, applying Theorem \ref{th:3.5} to Sylow $p$-subgroups of $G$, we obtain the following theorem:

\begin{theorem}\label{th:3.6} 
Let $G$ be a finite group of odd order and $p_1, \ldots, p_k$ be the set of primes dividing order of $G$. Let $P_i$ be a Sylow $p_i$-subgroup of $G$ and $S_i$ be a $p_i$-Sylow subgroup of $\operatorname{Aut}(P_i)$ with $\exp(S_i)=q_i$, $i=1, \ldots, k$. 
\begin{itemize} 
\item [$(i)$] If $P_i$ is metabelian $p_i$-group of class at most $2p_i-1$ for all $i=1, \ldots k$, then $\exp(G)\mid \prod_{i=1}^{k}p_iq_i^3$.
\item [$(ii)$] If the class of $P_i$ is $c_i$ for all $1\le i\le k$, then $\exp(G)\mid \prod_{i=1}^{k}p_i^{\ceil{\log_{p_i}c_i}}q_i^3$.
\end{itemize}
\end{theorem}

\section*{Acknowledgements} V. Z. Thomas acknowledges research support from SERB, DST, Government of India grant MTR/2020/000483.

\bibliographystyle{elsarticle-num}
\bibliography{Bibliography}
\end{document}